\documentclass[reqno,12pt]{amsart}

\usepackage{amssymb}
\usepackage{lineno}
\usepackage{hyperref} 
\usepackage{graphicx}
\usepackage{mathrsfs}
\usepackage{amsmath,amsthm,amssymb,enumerate}
\usepackage{color}
\newtheorem{thm}{Theorem}
\newtheorem*{thmA}{Theorem~A}
\newtheorem*{thmB}{Theorem~B}
\newtheorem*{thmB'}{Theorem~B$'$}
\newtheorem*{thmC}{Theorem~C}

\newtheorem{cor}{Corollary}
\newtheorem{lem}{Lemma}
\newtheorem*{lemA}{Lemma~A}

\newtheorem{rmk}{Remark}

\newtheorem*{conjB}{Bank's Conjecture}
\modulolinenumbers[5]

\begin{document}




\title{Hypertranscendency of Perturbations of Hypertranscendental Functions}


\author{Jiaxing Huang}
\address{Department of Mathematics, The University of Hong Kong, 
Pokfulam Road, Hong Kong}
\email{hjxmath@gmail.com}


\author{Tuen-Wai Ng}
\address{Department of Mathematics, The University of Hong Kong, 
Pokfulam Road, Hong Kong}
\email{ntw@maths.hku.hk}

\begin{abstract}
Inspired by the work of Bank on the hypertranscendence of $\Gamma e^h$ where $\Gamma$ is the Euler gamma function and $h$ is an entire function,
 we investigate when a meromorphic function $fe^g$ cannot satisfy any algebraic differential equation over certain field of meromorphic functions, 
 where $f$ and $g$ are meromorphic  and entire on the complex plane, respectively. Our results (Theorem \ref{cor:Bank} and \ref{cor:Bank2}) give  partial solutions to Bank's Conjecture (1977) on the hypertranscendence of $\Gamma e^h$. We also give some sufficient conditions for hypertranscendence of meromorphic function of the form $f+g$, $f\cdot g$ and $f\circ g$ in Theorem \ref{thm:fg} and \ref{thm:h}.

\end{abstract}









\maketitle

\section{Introduction and main results}\label{sec:I}

A meromorphic function $f$ on the complex plane is said to be \emph{hypertranscendental} over a field $\mathcal{K}$ of meromorphic functions, if $f$ does not satisfy any nontrivial algebraic differential equation whose coefficients are in the field $\mathcal{K}$. 
We are interested in those $\mathcal{K}$ which are related to the growth of $f$. Let $T(r, f)$ be the Nevanlinna characteristic function of $f$ (see Section \ref{sec:N} for the definitions and notations in Nevanlinna theory). 
We denote by $S(r, f)$ any quantity  which is of growth $o(T(r, f))$ as $r\rightarrow\infty$ outside a set of finite measure $E\subset(0, \infty)$.
By $\mathcal{M}_0$ we mean the field of meromorphic functions $y$ with $T(r, y)=o(r)$ as $r\rightarrow\infty$ outside a set of finite measure 
and $\mathcal{S}_f$ (resp. $\mathcal{S}^f$) the field of meromorphic functions $y$ satisfying the growth condition $T(r, y)=S(r, f)$ (resp. $T(r, y)=O(T(r, f))$ as $r\rightarrow\infty$ outside a set of finite measure).\\
 
 

In 1887, H\"{o}lder \cite{Holder87} established the hypertranscendence of the Euler gamma function $\Gamma$ over the field of rational functions, i.e., $\Gamma$ cannot satisfy any nontrivial algebraic differential equation whose coefficients are rational functions.
Hilbert \cite{Hilbert01}, in 1901, proved the hypertranscendence of Riemann zeta function using the functional equation of $\zeta$ and $\Gamma$. 
 In 1976,  Bank and Kaufman \cite{Bank76} extended the famous theorems of  H\"{o}lder and Hilbert by showing that $\Gamma$ and $\zeta$ are hypertranscendental over the field $\mathcal{M}_0$. 
One year later, Bank \cite{Bank77} asked to what extend the hypertranscendence of $\Gamma$ is due to the nature of its poles and zeros. In particular, he posed the following conjecture.

\begin{conjB}[\cite{Bank77}]\label{conj:Bank}
For every entire function $h$,
$\Gamma e^h$ is hypertranscendental over $\mathcal{M}_0$.
\end{conjB}

Bank \cite{Bank77} gave an affirmative answer to the above conjecture when either $h$ or $h'$ has only finitely many zeros. In 1980, he \cite{Bank80} generalized this result to the following.

 \begin{thmA}[\cite{Bank80}]\label{thm:A}
 Let $h$ be an entire function with the property that for some nonnegative integer $j$, 
and some complex number $a$, the following condition holds : 
\begin{equation}\label{eqn:B}
\overline N \big (r, 1/(h^{(j)} - a)\big ) = S(r, h^{(j)}),
\end{equation}
where as usual, $h^{(0)}$ denotes $h$. Then the function $\Gamma e^{h}$ is hypertranscendental over $\mathcal{M}_0$. 
 \end{thmA}
 
Related to Theorem A, we obtained the following. 

\begin{thm}\label{cor:Bank} Let $h$ be an entire function such that $T(r, \Gamma'/\Gamma)=S(r, h^{(j)})$ and 
\begin{equation}\label{eqn:hj}
\delta(a, h^{(j)})>0,
\end{equation}  for some $a\in\mathcal{M}_0$ and some nonnegative integer $j$. Then
$\Gamma e^h$ is hypertranscendental over $\mathcal{M}_0$.
\end{thm}

Related to Bank's Conjecture, we have the following partial result.

\begin{thm}\label{cor:Bank2} For any entire function $h$, $P(z, \Gamma e^h \dots, (\Gamma e^h)^{(n)})\not\equiv 0$ for any nontrivial distinguished polynomial $P(z, u_0, \dots, u_n)$ over $\mathcal{M}_0$.
\end{thm}

\begin{rmk} The notion of distinguished polynomial was first introduced by B. Q. Li and Z. Ye in \cite{LY16}. The definition is given as follow. 

 Let $I=(i_0, i_1, \dots, i_k)$ be a multi-index with $|I|=i_0+i_1+\cdots+i_k$. A polynomial in the variables $u_0, u_1, \dots, u_k$ with meromorphic function coefficients in a set $S$ can always be written as $$P(z, u_0, u_1, \dots, u_k)=\sum_{I\in\Lambda}a_I(z)u_0^{i_0}u_1^{i_1}\cdots u_k^{i_k},$$ where the coefficients $a_I$ are meromorphic functions in $S$ and $\Lambda$ is an index set. 
 We call $P$ a \emph{distinguished polynomial} in $u_0, u_1, \dots, u_k$ with coefficients in $S$, or simply an \emph{$S$-distinguished polynomial}, if the index set $\Lambda$ satisfies $|I_i|\neq |I_j|$ for any distinct indices $I_i, I_j$ in $\Lambda$. In other words, each homogeneous part of the distinguished polynomial $P$ contains one term only.

\end{rmk}





 If $\mathcal{K}$ is a field of meromorphic functions, we denote by $A(\mathcal{K})$ the set of all meromorphic functions which satisfy some algebraic differential equation over $\mathcal{K}$. It is well known (see Chapter 14 of \cite{IIpo11}) that $A(\mathcal{K})$ is a differential field, i.e., a field with an additional map $D: A(\mathcal{K})\rightarrow A(\mathcal{K})$ such that $D(a\cdot b)=(Da)\cdot b+a\cdot Db$ for any $a, b\in A(\mathcal{K})$.\\

To explain the difference between Theorem A and Theorem \ref{cor:Bank},
let us sketch the main idea of the proof of Theorem A (see Part B in \cite{Bank80} or Chapter 14 of \cite{IIpo11}). \\

Let $h$ be an entire function satisfying the assumption (\ref{eqn:B}) in Theorem A.
 If $\Gamma e^h\in A(\mathcal{M}_0)$ and $a\in\mathbb{C}$, set $g=h- (az^j/j!)$ which satisfies the condition $\overline N \big (r, 1/g^{(j)}\big ) = S(r, g^{(j)})$. Applying Lemma A below and using the fact that $T(r, \Gamma'/\Gamma)=r+o(r)$,  one can conclude that $T(r, g^{(j)})=O(r)$. On the other hand, $g$ is an entire function with $\overline N \big (r, 1/g^{(j)}\big ) = S(r, g^{(j)})$, thus $T(r, g^{(j+1)}/g^{(j)})=o(r)$. Hence $g^{(j+1)}/g^{(j)}$ belongs to $\mathcal{M}_0$ which implies $g\in A(\mathcal{M}_0)$. Thus $h\ \mbox{and}\ h'\in A(\mathcal{M}_0)$. Since $h'=(e^h)'/(e^h)$, it follows that $e^h\in A(\mathcal{M}_0)$. Combining with the assumption that $\Gamma e^h\in A(\mathcal{M}_0)$, one can deduce a contradiction to the hypertranscendence of $\Gamma$ over $\mathcal{M}_0$. \\
 
 Actually, from the proof of Theorem A, it is not hard to see that the assumption $\Gamma e^h\in A(\mathcal{M}_0)$ and the condition (\ref{eqn:B}) imply that $T(r, h^{(j)})=O(T(r, \Gamma'/\Gamma))$. Our Theorem \ref{cor:Bank} considers a sort of complement assumption that $T(r, \Gamma'/\Gamma)=S(r, h^{(j)})$.\\
 Under this assumption, the condition (\ref{eqn:hj}) is less restrictive than the one on $\overline N(r, 1/(h^{(j)}-a))$ in Theorem A. In addition, $a$ can also be nonconstant.\\

To produce more examples of hypertranscendental functions, Bank also investigated the hypertranscendency of the perturbation of hypertranscendental  meromorphic functions  by adding a small function.

 \begin{thmB}[\cite{Bank80}]\label{thm:B}
 Let $f$ be a meromorphic function on the complex plane which is hypertranscendental over a differential field $\mathcal{S}\subset\mathcal{S}_f$ . Let $g$ be a meromorphic function on the complex plane. Then, if $f+g$ satisfies an algebraic differential equation over $\mathcal{S}$, we have $$T(r, f)=O(\overline N(r, 1/f)+\overline N(r, f) +T(r, g))$$ as $r\rightarrow\infty$ outside of a possible exceptional set of finite measure. 
 
 In particular, if all $\overline N(r,1/f), \overline N(r, f)$ and $T(r, g)$ are $S(r, f)$, then $f+g$ must be hypertranscendental over $\mathcal{S}$.
 \end{thmB}
 
 
The proofs of Theorem A and B in \cite{Bank77, Bank80} depend on the following Lemma first appeared in \cite{Bank761}.
 
\begin{lemA}[\cite{Bank761}]
Let $P(z, y, y', \dots, y^{(n)})$ be a polynomial in $y, y', \dots, y^{(n)}$ whose coefficients are meromorphic functions on $\mathbb{C}$. For each $r>0$, let $\Delta(r)$ be the maximum of the Nevanlinna characteristics of the coefficients of $P$. Let $f$ be a nonzero meromorphic function on the complex plane satisfying the equation $P=0$, but for some nonnegative integer $q$, $P_q(f, f', \dots, f^{(n)})\neq 0$, where $P_q$ is the homogeneous part of $P$ of total degree $q$ in the indeterminates $y, y', \dots, y^{(n)}$. Then 
$$T(r, f)=O(E(r)),$$ as $r\to\infty$, 
outside of a possible exceptional set of finite measure, 
where $$E(r)=\overline N(r, 1/f)+\overline N(r, f) +\Delta(r)+\log r.$$ In addition, for any $\alpha>1$, there exist positive constants $c$ and $r_0$ such that $$T(r, f)\leq cE(\alpha r), \ \mbox{for all}\ r\geq r_0.$$
\end{lemA}

In 1991, Y. Z. He and C. C. Yang \cite{HY91} proved that $\Gamma(g)$ is hypertranscendental over the field $\mathcal{M}^g$ of meromorphic functions $y$ with $T(r, y)=O(T(r, g))$ by using Steinmetz's Reduction Theorem (Theorem C below). Their method can be applied to the general case (see Theorem \ref{thm:fg}).  In 2007,  Markus \cite{Markus07} applied the method of differential algebra to obtain the hypertranscendence of $\zeta(\sin z)$ and $\Gamma(\sin z)$ over $\mathbb{C}$, and he proved  the differential independence between $f_i$ and $f_j(\sin z)$ for $i, j =1, 2$, where $f_1=\Gamma$ and  $f_2=\zeta$.\\

Applying the same idea of He and Yang in \cite{HY91}, we obtain the following general result which covers the results of He and Yang \cite{HY91}.
\begin{thm}\label{thm:fg}
Let $f$ be hypertranscendental over the rational function field $\mathbb{C}(z)$ and $g$ be a nonconstant entire function. Then $f\circ g$ is hypertranscendental over the field $\mathcal{S}^g$.
\end{thm}

As a consequence, we can generalize a result of L. Markus (see Lemma 1 in \cite{Markus07}) by using a different method.
\begin{cor} Let $a$ be a nonzero complex number. Then
both $\Gamma(\sin az)$ and $\zeta(\sin az)$ are hypertranscendental over the field of meromorphic functions $y$ with $T(r, y)=O(r)$ as $r\rightarrow\infty$ outside some set of finite measure.
\end{cor}


It is natural to consider the hypertranscendency of $g\circ f$ over some fields for entire hypertranscendental $f$ and meromorphic $g$. This seems to be a more difficult problem as  Steinmetz's Reduction Theorem cannot be applied directly here (see Remark \ref{rmk:fg} in Section \ref{sec:fg}). However, we do obtain one related result in Theorem \ref{thm:h}(\ref{t1}).\\

 Inspired by the results of Bank, He-Yang and Markus,  in this paper, we will first prove a result similar to Lemma A,  that is  $T(r, f)$ can be controlled by one counting function 
$N(r, 1/f)$ (see Lemma \ref{lem:sm}). Using Lemma \ref{lem:sm}, we then  obtain the following results on the hypertranscendency of perturbations of hypertranscendental functions, including that of $\Gamma$ and $\Gamma e^h$.

 \begin{thm}\label{thm:h} 
Let $g$ and $f$ be meromorphic functions and $\mathcal{S}$ be the field of meromorphic functions $y$ with $T(r, y)=S(r, f'/f)$, i.e. $\mathcal{S}=\mathcal{S}_{f'/f}$. Let $\mathcal{O}$ be the set of entire functions on $\mathbb{C}$. Suppose $f$ is hypertranscendental over $\mathcal{S}$. 
\begin{enumerate}[(1)]

\item\label{t1} If $f\in\mathcal{O}$, and $g-R$ has finitely many zeros, where $R$ is a non-constant rational function, then $g\circ f$ is  hypertranscendental over $\mathcal{S}$.

\item\label{t2} Assume that  $f\in\mathcal{S}_g$ and $\delta(a, g)>0
$ for some $a\in\mathcal{S}\setminus\{0\}$, then $fg$ is hypertranscendental over $\mathcal{S}$. 
 
\item\label{t2.1} If there exists a non-negative integer $k$ such that $T(r, f)=S(r, g^{(k)})$ and $\delta(a, g^{(k)})>0$
for some $a\in\mathcal{S}$, then $f+g$ is hypertranscendental over $\mathcal{S}$. 

\item\label{t3} Assume that $g\in\mathcal{O}$, and if there exists a nonnegative integer $k$ such that $T(r, f'/f)=S(r, g^{(k)})$ and \begin{equation}\label{eqn:ng}
\delta(a, g^{(k)})>0
\end{equation} for some $a\in\mathcal{S}$, then $fe^g$ is  hypertranscendental over $\mathcal{S}$.
\item\label{t4} If $g\in\mathcal{O}$ and $f\in\mathcal{S}_{\exp(g)}$ , then $P(z, fe^g, (fe^g)', $ $\dots, (fe^g)^{(n)})\not\equiv 0$ for any nontrivial distinguished polynomial $P(z, u_0, \dots, u_n)$ over $\mathcal{S}$.

\end{enumerate}

\end{thm}

In Section \ref{sec:12}, we will use Theorem \ref{thm:h} to prove Theorem \ref{cor:Bank} and \ref{cor:Bank2}. Section \ref{sec:N} introduces the basics of Nevanlinna Theory. Theorem \ref{thm:fg} and \ref{thm:h} will be proven in Section \ref{sec:fg} and \ref{sec:P}, respectively.

\section{Nevanlinna Theory}\label{sec:N}
We recall the basic notations and results of Nevanlinna theory \cite{IIpo11} which are main tools for proving our results.\\

 Let $f$ and $a$ be meromorphic functions in the complex plane $\mathbb{C}$ and $\mathbb{D}_r=\{|z|<r\}$. Denote the number of poles of $f$ in $\mathbb{D}_r$ by $n(r, f)$, and let $n(r, a)=n(r, a, f)=n(r, 1/(f-a))$. When the number of distinct poles of $f$ in $\mathbb{D}_r$ is denoted by $\overline{n}(r, f)$, we then let $\overline{n}(r, a)=\overline{n}(r, 1/(f-a))$. Correspondingly we define the counting function and truncated counting function in Nevanlinna theory as follows: 
 $$N(r, a, f):=\displaystyle\int_0^r\dfrac{n(t, a)-n(0, a)}{t}dt+n(0,a)\log r;$$
 $$\overline{N}(r, a, f):=\displaystyle\int_0^r\dfrac{\overline{n}(t, a)-\overline{n}(0, a)}{t}dt+\overline{n}(0,a)\log r.$$
 The proximity function is defined as 
 $$m(r, f):=\dfrac{1}{2\pi}\displaystyle\int_0^{2\pi}\log^+\left|f(re^{i\theta})\right|d\theta$$ and 
 $$m(r, a, f):=m(r, 1/(f-a)),$$ where $\log^+x=\max\{0, \log x\}$. The Nevanlinna characteristic function of $f$ is defined by $$T(r, f)=m(r, f)+N(r, f).$$ The First Main Theorem of Nevanlinna theory for small functions \cite{Mohon82} says that
for any meromorphic function $a$ with $T(r,a)=S(r, f)$, $$T(r, f)=T(r, a, f)+S(r, f)$$ where $T(r, a, f):=m(r, a, f)+N(r, a, f)$.
Finally, we denote the Nevanlinna order of $f$ by $$\rho(f):=\limsup_{r\rightarrow\infty}\frac{\log T(r, f)}{\log r},$$ and the deficiency of $a$ for $f$ by $$\delta(a, f):=\liminf_{r\rightarrow\infty}\frac{m(r, a, f)}{T(r, f)}=1-\limsup_{r\rightarrow\infty}\frac{N(r, a, f)}{T(r, f)}.$$ If $\delta(a, f)>0$, then we say $a$ is a  deficient function of $f$.\\


The logarithmic derivative lemma states that 
 \begin{lem}[\cite{IIpo11}]\label{lem:LDL}
 Let $f$ be a transcendental meromorphic function and $k\geq 1$ be an integer. Then $$m\left(r, \frac{f^{(k)}}{f}\right)=S(r, f),$$ and if $f$ is of finite order of growth, then $$m\left(r, \frac{f^{(k)}}{f}\right)=O(\log r).$$
 \end{lem}

\section{Proof of Theorem \ref{thm:fg}}\label{sec:fg}
To prove Theorem \ref{thm:fg}, we first introduce the Steinmetz's Reduction Theorem. 

\begin{thmC}[Steinmetz's Reduction Theorem \cite{GO89, Stem80}]\label{thm:Stein}
Let $F_j, 1\leq j\leq N$ be meromorphic functions on $\mathbb{C}$. Let $h_j, 1\leq j\leq N$ be meromorphic and $g$ be entire on $\mathbb{C}$ such that for each $j$, $$T(r, h_j)=O(T(r, g))$$ as $r\rightarrow\infty$ outside some set of finite measures.
 Given a functional equation of the form $$F_1(g(z))h_1(z)+\cdots+F_N(g(z))h_N(z)=0,$$ then there exist polynomials $p_j$, not all zeros, such that $$p_1(g(z))h_1(z)+\cdots+p_N(g(z))h_N(z)=0.$$ Furthermore, if $h_j\not\equiv 0$ for some $j$, then there exist polynomials $Q_j$, not all zeros, such that $$F_1(z)Q_1(z)+\cdots+F_N(z)Q_N(z)=0.$$
\end{thmC}

\begin{proof}[Proof of Theorem \ref{thm:fg}] We will follow the idea of the proof of Theorem 4 in \cite{HY91}.

Suppose that $f\circ g$ satisfies a nontrivial algebraic differential equation with coefficients in $\mathcal{S}^g$, i.e., there exists a nontrivial differential polynomial $P(z, w, w', \dots, w^{(n)})$ with coefficients in $\mathcal{S}^g$ such that
$$P(z, f\circ g, (f\circ g)', \dots, (f\circ g)^{(n)})=\sum_{j}(M_j(f)\circ g)(H_j(g)(z))=0$$ where $M_j(f)$ is a differential monomial of $f$ with constant coefficients and $H_j(g)(z)$ is a differential polynomial of $g(z)$ whose coefficients are some linear combinations of the coefficients of the original differential polynomial $P(z, w, w', \dots, w^{(n)})$. 

Now, set $F_j(z)=M_j(f)(z)$ and $h_j(z)=H_j(g)(z)$, it follows from the second result of Theorem $\mathrm{C}$ that there exist polynomials $Q_i$, not all zeros, such that$$F_1(z)Q_1(z)+\cdots+F_N(z)Q_N(z)=0.$$
which implies that $f$ satisfies a nontrivial algebraic differential equation with coefficients in $\mathbb{C}(z)$. This is a contradiction to our assumption that $f$ is hypertranscendental over $\mathbb{C}(z)$.

\end{proof} 

\begin{rmk}\label{rmk:fg}Here, we will explain the reason why the Steinmetz's Reduction Theorem does not work for the hypertranscendency of $g\circ f$.
We use the same idea of proof of Theorem \ref{thm:fg}. Suppose that $g\circ f$ satisfies a nontrivial algebraic differential equation over a suitable field such that we can apply the Steinmetz's Reduction Theorem, thus we have  
$$p_1(f(z))h_1(z)+\cdots+p_N(f(z))h_N(z)=0$$
or $$F_1(z)Q_1(z)+\cdots+F_N(z)Q_N(z)=0$$
where $h_j(z)$ is a differential polynomial of $f(z)$ whose coefficients are some linear combinations of the coefficients of the algebraic differential equation $g\circ f$ satisfied, and $F_j(z)$ is a differential monomial of $g$ with constant coefficients.
From these two equalities, we cannot deduce any contradictions even through we have known the hypertranscendency of $f$.
\end{rmk}

 \section{Proof of Theorem \ref{thm:h}}\label{sec:P}
 In this section, we are now going to prove our main result (Theorem \ref{thm:h}). To prove Theorem \ref{thm:h},  we need the following lemmata.


\begin{lem}\label{lem:sm}
Let $f$ be a nonzero meromorphic function on the complex plane and $P(z, y, y', \dots, y^{(n)})$ be a polynomial in $y, y', \dots, y^{(n)}$ whose coefficients are in the field $\mathcal{S}_f$. Suppose $f$ satisfies the equation $P=0$. Rewrite $P=0$ as $P_q=\sum_{j=k}^mP_j$, for some nonnegative integers $q$ and $k(>q)$ such that $P_q\neq 0$ for each $j\geq k$, where $P_j$ is the homogeneous part of $P$ of total degree $j$ in the indeterminates $y, y', \dots, y^{(n)}$. Then for any integer $N$ with $q\leq N\leq k$,  $$m(r, P_q/f^N)=S(r, f).$$ In addition if $q=0$, then $$T(r, f)=N(r, 0, f)+S(r, f).$$
\end{lem}

\begin{rmk} Lemma \ref{lem:sm} is essentially B.Q. Li's Lemma 4.1 in \cite{Li96}
\end{rmk}

\begin{proof}[Proof of Lemma \ref{lem:sm}] 
Let $P(z, u_0, \dots, u_n)$ be a polynomial in $u_0, \dots, u_n$ with coefficients in $\mathcal{S}_f$. 
Assume that $$I=\{i:=(i_0,i_1,  \dots, i_n)|\ i_j\ \mbox{is  a nonnegative integer and}\ 0\leq j\leq n\}$$ is an index set with finite cardinal numbers. 
Define $$|i|=\sum_{j=0}^ni_j\quad \mbox{and}\quad I_p=\{i\in I: |i|=p\}.$$ 
For each $l\geq q$, let $$P_l=\sum_{i\in I_l}a_i(z)u_0^{i_0}\dots u_n^{i_n}$$
where $a_i\in\mathcal{S}_f$.

Take any point $z\in\mathbb{C}$, we consider several cases. 

\noindent \textbf{Case (i)} $|f(z)|\geq 1$. Since $P_q=\displaystyle\sum_{i\in I_q}a_i(z)u_0^{i_0}\dots u_n^{i_n}$, 
\begin{eqnarray*} \left|\frac{P_q(f, f',  \dots, f^{(n)})}{f^N}(z)\right|
&\leq&\left|\frac{P_q(f,f' \dots, f^{(n)})}{f^q}(z)\right|\\
&\leq& \sum_{i\in I_q}\left|a_i(z)\frac{f^{i_0}(f')^{i_1}\cdots (f^{(n)})^{i_m}}{f^q}\right|:=G_1(z).
\end{eqnarray*}

\noindent \textbf{Case (ii)} $|f(z)|\leq 1$. Then by $P_q=\displaystyle\sum_{j=k}^mP_j, q\leq N\leq k$, we have 
\begin{eqnarray*} \left|\frac{P_q(f, f',  \dots, f^{(n)})}{f^N}(z)\right|
&=&\left|\sum_{j=k}^m\frac{P_j(f,f' \dots, f^{(n)})}{f^j}(z)f^{j-N}\right|\\
&\leq&\sum_{j=k}^m \left|\frac{P_j(f,f' \dots, f^{(n)})}{f^j}(z)\right| |f|^{j-N}\\
&\leq& \sum_{j=k}^m\sum_{i\in I_j}\left|a_i(z)\frac{f^{i_0}(f')^{i_1}\cdots (f^{(n)})^{i_m}}{f^j}\right|:=G_2(z).
\end{eqnarray*}
Combining the above results, we see that in any case $$\left|\frac{P_q(f, f',  \dots, f^{(n)})}{f^N}(z)\right|\leq G_1(z)+G_2(z)$$ for any $z\in\mathbb{C}$. By the well-known Logarithmic Derivative Lemma and $a_i\in\mathcal{S}_f$, we deduce that $$m(r, P_q/f^N)\leq m(r, G_1+G_2)=S(r, f).$$
Now if $q=0$, then by taking $N=1$, we have $$m(r, 1/f)\leq m(r, P_0/f)+m(r, 1/P_0)+O(1)=S(r, f)$$ as $T(r, P_0)=S(r, f)$. Hence the result follows from the First Main Theorem of Nevanlinna theory. 
\end{proof}

As a consequence, one can also obtain the following lemma first proved by A. Mohon'ko in 1982. 
\begin{lem}[\cite{Mohon82}]\label{thm:C} Let $f$ be a transcendental meromorphic solution of an algebraic differential equation $P(y)=P(z, y, y', \dots, y^{(k)})=0$ with coefficients in $\mathcal{S}_f$. If a meromorphic function $\phi$ with $T(r, \phi)=S(r, f)$ does not solve $P(z, y, y', \dots, y^{(k)})=0$ i.e. $P(z, \phi, \phi', \dots, \phi^{(k)})\not\equiv 0$, then $$m\left(r, \frac{1}{f-\phi}\right)=S(r, f)$$
\end{lem}
\begin{proof} Let $g=f-\phi$, then $T(r, g)=T(r, f)+S(r, f)$. Since $P(f)\equiv 0$, we have $$P(f)=P(g+\phi)=Q(g)+P(\phi)\equiv 0$$ where $Q$ is a differential polynomial over $\mathcal{S}_f$ with lowest degree at least one, as $T(r, \phi)=S(r, f)$. The result follows immediately from Lemma \ref{lem:sm} as $P(\phi)\not\equiv 0$.
\end{proof}
\begin{lem}[\cite{Clunie70}]\label{lem:Cl}
Let $f$ be a transcendental entire function and let $g$ be a transcendental meromorphic function in the complex plane, then $T(r, f)=o(T(r, g\circ f))$ as $r\rightarrow\infty$.
\end{lem}

\begin{proof}[Proof of Theorem \ref{thm:h}]
 \ref{t1}). Without loss of generality, we can assume $R(z)=z$, since if $f$ is hypertranscendental over $\mathcal{S}$, it is easy to show that $R\circ f$ is also hypertranscendental over $\mathcal{S}$. 

Suppose $g(z)-z=0$ has $d$ roots, then $g(z)-z=Q(z)A(z)$ where $Q$ is a polynomial with degree $d$, and $A$ is a transcendental meromorphic function which is nowhere zero. Hence 
 if $f$ is an entire function, we have  $$N(r, 0, g\circ f- f)=N(r, 0, Q(f)A(f))=N(r, 0, Q(f))\leq dT(r, f)+S(r, f).$$
By  Lemma \ref{lem:Cl}, we have $T(r, f)=o(T(r, g\circ f))$. Suppose $g\circ f$ is not hypertranscendental over $\mathcal{S}$, that is, $g\circ f$ is a solution of an algebraic differential equation $P(z, y, y', \dots, y^{(k)})=0$ with coefficients in $\mathcal{S}$ (hence in $\mathcal{S}_{g\circ f}$ as well). 
 By Lemma \ref{thm:C} and the assumption that $f$ is hypertranscendental over $\mathcal{S}$, we have $$m\left(r, \dfrac{1}{g\circ f- f}\right)=S(r, g\circ f).$$ 
By the First Main Theorem of Nevanlinna Theory for small functions \cite{Mohon82},
\begin{eqnarray*}
T(r, g\circ f)&=&T(r, g\circ f-f)+S(r, g\circ f)\\
&=&m(r, 0, g\circ f-f)+N(r, 0, g\circ f- f)+S(r, g\circ f)\\
&\leq& S(r, g\circ f)+dT(r, f)=S(r, g\circ f)
\end{eqnarray*} which is a contradiction. This completes the proof of the first part. 
 
 \ref{t2}). If $a\not\equiv 0$, since $f$ is hypertranscendental over $\mathcal{S}$, it is easy to show that $af$ is also hypertranscendental over $\mathcal{S}$, as $a\in\mathcal{S}$.
 
Since $T(r, f)=S(r, g), T(r, a)=S(r, f'/f)=S(r, f)$, one can obtain that $T(r, af)=T(r, f)+S(r, f)=S(r, fg).$

Suppose $fg$ is not hypertranscendental over $\mathcal{S}$, that is, $fg$ is a solution of an algebraic differential equation $P(z, y, y', \dots, y^{(k)})=0$ with coefficients in $\mathcal{S}$ (hence in $\mathcal{S}_{fg}$ also). By Lemma \ref{thm:C} and the hypertranscendence of $af$ over $\mathcal{S}$ , we have $$m\left(r, \dfrac{1}{fg-af}\right)=S(r, fg)=S(r, g).$$ 
On the other hand, by the First Main Theorem of Nevanlinna Theory for small functions, as $T(r, af)=S(r, fg)$,
 \begin{eqnarray*}
T(r, fg)&=&T(r, fg-af)+S(r, fg)\\
&=&m(r, 0, fg-af)+N(r, 0, fg - af)+S(r, fg)\\
&\leq& N(r, 0, g-a)+N(r, 0, f)+S(r, g)\\
&=& N(r, 0, g-a)+S(r, g). 
\end{eqnarray*} Since $T(r, fg)=T(r, g)+S(r, g)$, it follows that $T(r, g)=N(r, a, g)+S(r, g)$ which is a contradiction to the assumption that $\delta(a, g)>0$.
 
\ref{t2.1}). If $f+g \in A(\mathcal{S})$, so does $f^{(k)}+g^{(k)}$, that is, there exists a nontrivial algebraic differential equation $P(z, y, y', \dots, y^{(n)})=0$ over $\mathcal{S}$ such that $$P(z, f^{(k)}+g^{(k)}, f^{(k+1)}+g^{(k+1)}, \dots, f^{(k+n)}+g^{(k+n)})\equiv 0.$$ Set $$Q(z, g^{(k)}, g^{(k+1)}, \dots, g^{(n+k)}):=P(z, f^{(k)}+g^{(k)}, f^{(k+1)}+g^{(k+1)}, \dots, f^{(k+n)}+g^{(k+n)}),$$ 
then $Q(z, g^{(k)}, g^{(k+1)}, \dots, g^{(n+k)})\equiv 0$. It is easy to check that all the Nevanlinna characteristic functions of the coefficients of $Q(z, g^{(k)}, g^{(k+1)}, \dots, g^{(n+k)})$ are $S(r, g^{(k)})$, as $T(r, f)=S(r, g^{(k)})$ and $T(r, f^{(k)})\leq (k+1)T(r, f)+S(r, f)$

On the other hand,  since $f$ is hypertranscendental over $\mathcal{S}$, so is $f^{(k)}+a$ for any $a\in\mathcal{S}$, hence $$Q(z, a, a' , \dots, a^{(n)})=P(z, f^{(k)}+a, f^{(k+1)}+a', \dots, f^{(k+n)}+a^{(n)})\not\equiv 0.$$ By Lemma \ref{thm:C}, we have $$m\left(r, \frac{1}{g^{(k)}-a}\right)=S(r, g^{(k)})$$ which is a contradiction to the assumption that $\delta(a, g^{(k)})>0$ for some $a\in\mathcal{S}$. 

\ref{t3}).  If $fe^g\in A(\mathcal{S})$, then clearly, $\dfrac{f'}{f}+g'=\dfrac{(fe^g)'}{fe^g}\in A(\mathcal{S})$, and hence so does $\left(\dfrac{f'}{f}\right)^{(k)}+g^{(k+1)}$ for any nonnegative integer $k$, 
that is, there exists an algebraic differential equation $P(z, y, y', \dots, y^{(n)})=0$ over $\mathcal{S}$ such that 
$$P\left(z, \left(\dfrac{f'}{f}\right)^{(k)}+g^{(k+1)}, \left(\dfrac{f'}{f}\right)^{(k+1)}+g^{(k+2)}, \dots,\left(\dfrac{f'}{f}\right)^{(k+n)}+g^{(k+n+1)}\right)\equiv 0.$$ 
Set $$Q(z, g^{(k)}, g^{(k+1)}, \dots, g^{(n+k+1)}):=P\left(z, \left(\dfrac{f'}{f}\right)^{(k)}+g^{(k+1)},  \dots,\left(\dfrac{f'}{f}\right)^{(k+n)}+g^{(k+n+1)}\right),$$ then $$Q(z, g^{(k)}, g^{(k+1)}, \dots, g^{(n+k+1)})\equiv 0$$ 
and all the Nevanlinna characteristic functions of the coefficients of $Q(z, g^{(k)}, g^{(k+1)}, \dots,$ $g^{(n+k+1)})$ are $S(r, g^{(k)})$ from $$T(r, f'/f)=S(r, g^{(k)})$$ and $$T(r, (f'/f)^{(j)})\leq(j+1)T(r, f'/f)+S(r, f'/f)$$ for any nonnegative integer $j$.

On the other hand, since $f$ is hypertranscendental over $\mathcal{S}$, so is $(f'/f)^{(k)}$ for any nonnegative integer $k$, and hence so is $(f'/f)^{(k)}+a'$ for any $a\in\mathcal{S}$. Therefore, $$Q(z, a, a' ,\dots, a^{(n+1)})=P\left(z, \left(\dfrac{f'}{f}\right)^{(k)}+a', \left(\dfrac{f'}{f}\right)^{(k+1)}+a'', \dots,\left(\dfrac{f'}{f}\right)^{(k+n)}+a^{(n+1)}\right)\not\equiv 0.$$ By Lemma \ref{thm:C}, we have $$m\left(r, \frac{1}{g^{(k)}-a}\right)=S(r, g^{(k)})$$ which is a contradiction to the inequality (\ref{eqn:ng}).


 \ref{t4}).  Let $$P(z, u_0, u_1, \dots, u_n)=\sum_{i=0}^mP_i(z, u_0, u_1, \dots, u_n)$$ be a distinguished polynomial over $\mathcal{S}$, where $P_i(z, u_0, u_1, \dots, u_n)$ contains only one term $a_i(z)u_0^{i_0}u_1^{i_1}\cdots u_n^{i_n}$ with coefficient $a_i\in\mathcal{S}$ and $i=i_0+i_1+\cdots+i_n$. 

We first notice that the assumption $f\in\mathcal{S}_{\exp(g)}$ and Lemma \ref{lem:Cl} imply that 
\begin{eqnarray*} T\left(r, \dfrac{(fe^g)'}{fe^g}\right)
&=&T\left(r, \dfrac{f'}{f}+g'\right)\\
&\leq&2T(r, f)+S(r, f)+2T(r, g)+S(r, g)=S(r, fe^g)
\end{eqnarray*}

 Assume to the contrary that $P(z, fe^g, (fe^g)', \dots, (fe^g)^{(n)})\equiv 0$. Let $q$ be a nonnegative integer such that $a_q\neq 0$ and $a_j\equiv 0, j=0, 1, \dots, q-1$.  Applying Lemma \ref{lem:sm} to $N=q+1$,  one can conclude that $$m(r, P_q/(fe^g)^{q+1})=S(r, fe^g)$$ hence $m(r, 1/(fe^g))=S(r, fe^g)$ as $T(r, P_q/(fe^g)^q)=S(r, fe^g)$. However, $m(r, 0, fe^g)+N(r, 0, fe^g)=S(r, fe^g)+N(r, 0, f)\leq T(r, f)=S(r, fe^g)$, which is impossible, thus $a_q\equiv 0$. Repeating the above argument, one can obtain that $a_i\equiv 0$ for all $i=0, 1, \dots, m$. Hence the result follows.

This completes the proof of Theorem \ref{thm:h}.
\end{proof}

\section{Proof of Theorem \ref{cor:Bank} and \ref{cor:Bank2}}\label{sec:12}
In this section, we will prove Theorem \ref{cor:Bank} and \ref{cor:Bank2} by using Theorem \ref{thm:h}.
\begin{proof}[Proof of Theorem \ref{cor:Bank}]
 This follows immediately from part (\ref{t3}) of Theorem \ref{thm:h} and the fact that $T(r, \Gamma'/\Gamma)=r+o(r)$ in \cite{Bank76}.
 \end{proof}
 
\begin{proof}[Proof of Theorem \ref{cor:Bank2}]
 We consider the following two cases.

\noindent\textbf{Case 1}. If $\rho(e^h)<\infty$, then $\Gamma e^h$ is hypertranscendental over $\mathcal{M}_0$ (see p.271 of \cite{Bank80}). Actually, in this case, $h$ is a polynomial, hence it is not hard to see that $e^h\in A(\mathcal{M}_0)$ as $h'=(e^h)'/e^h$. If $\Gamma e^h\in A(\mathcal{M}_0)$, one can conclude that $\Gamma\in A(\mathcal{M}_0)$ which is a contradiction to the hypertranscendence of $\Gamma$ over $\mathcal{M}_0$.

\noindent\textbf{Case 2}. If $\rho(e^h)=\infty$, then $\Gamma\in\mathcal{S}_{\exp(h)}$, hence the result follows immediately from Theorem \ref{thm:h}(\ref{t4}).
\end{proof}

\section*{Acknowledgements} The first author was partially supported by a graduate studentship of HKU and the RGC grant 1731115. The second author was partially supported by the RGC grant 1731115. The authors are very grateful to the referee for the valuable suggestions.



\bibliographystyle{abbrv}
  \bibliography{mybibfile}

\begin{thebibliography}{10}

\bibitem{Bank761}
S.~Bank.
\newblock On the growth of meromorphic solutions of linear differential
  equations having arbitrary entire coefficients.
\newblock {\em Ann. Mat, Pura Appl.}, 107:279--289, 1976.

\bibitem{Bank77}
S.~Bank.
\newblock Some results on {G}amma function and other hypertranscendental
  functions.
\newblock {\em Proc. R. Soc. Edinb.}, 79A:335--341, 1977.

\bibitem{Bank80}
S.~Bank.
\newblock Some {R}esults on {H}ypertranscendental {M}eromorphic {F}unctions.
\newblock {\em Monatsh. Math.}, 90:267--289, 1980.

\bibitem{Bank76}
S.~Bank and R.~Kaufman.
\newblock An extension of {H}{\"o}lder's theorem concerning the {G}amma
  function.
\newblock {\em Funkcialaj Ekvacioj}, 19:53--63, 1976.

\bibitem{Clunie70}
J.~Clunie.
\newblock {\em The composition of entire and meromorphic functions}.
\newblock Mathematical Essays Dedicated to A. J. Macintyre. OH. Univ. Press,
  Athens, Ohio, 1970.

\bibitem{GO89}
F.~Gross and C.~F. Osgood.
\newblock A simple proof of a theorem of {S}teinmetz.
\newblock {\em J. Math. Anal. Appl.}, 143:290--294, 1989.

\bibitem{HY91}
Y.~Z. He and C.~C. Yang.
\newblock Meromorphic functions that do not satisfy algebraic differential
  equations.
\newblock {\em Acta Math. Sci. (Chinese)}, 11(3):343--348, 1991.

\bibitem{Hilbert01}
D.~Hilbert.
\newblock Mathematische probleme.
\newblock {\em Arch.Math.Phys.}, 1:44--63, 213--317, 1901.

\bibitem{Holder87}
O.~H{\"o}lder.
\newblock {\"U}ber die {E}igenschaft der {$\Gamma$}-{F}unction, keiner
  algebraischen {D}ifferentialgleichung zu gen{\"u}gen.
\newblock {\em Math. Ann.}, 28:1--13, 1887.

\bibitem{IIpo11}
I.~Laine.
\newblock {\em Nevanlinna theory and complex differential equations}.
\newblock Walter de Gruyter, Berlin, 1993.

\bibitem{Li96}
B.~Q. Li.
\newblock On reduction of functional-differential equations.
\newblock {\em Complex Variables}, 31:311--324, 1996.

\bibitem{LY16}
B.~Q. Li and Z.~Ye.
\newblock Algebraic differential equations with functional coefficients
  concerning {$\zeta$} and {$\Gamma$}.
\newblock {\em J. Differential Equations}, 260:1456--1464, 2016.

\bibitem{Markus07}
L.~Markus.
\newblock Differential independence of {$\Gamma$} and {$\zeta$}.
\newblock {\em J. Dynam. Differential Equations}, 19:133--154, 2007.

\bibitem{Mohon82}
A.~Z. Mokhon'ko.
\newblock Estimates of {N}evanlinna characteristics of algebroid functions and
  their applications to differential equations.
\newblock {\em Sib. Math. J.}, 23:80--88, 1982.

\bibitem{Stem80}
N.~Steinmetz.
\newblock \"{U}ber die {F}aktorisierbaren {L}osungen gew\"ohnlicher
  differential {G}leichungen.
\newblock {\em Math. Z.}, 170:169--180, 1980.

\end{thebibliography}





\end{document}